\documentclass[12pt,reqno]{amsart}
\usepackage{amsmath,amssymb,latexsym,amscd,amsthm} % ams auxiliaries.
\usepackage{graphicx}
\usepackage{booktabs}

\usepackage[tamon,bojan]{optional} %% Annotation for draft versions.

\usepackage{xcolor}

\newtheorem{theorem}{Theorem}[section]  % If subsection, get 3.1.2, etc.

\theoremstyle{definition}

\theoremstyle{remark}

\title{Expected Crossing Numbers}
\author[Bojan Mohar]{Bojan Mohar}
\address{Department of Mathematics, Simon Fraser University,
8888 University Drive, Burnaby, British Columbia V5A 1S6, Canada}
\author[Tamon Stephen]{Tamon Stephen}

\begin{document}

%%%%%%%%%%%%%%%%%%%%%%%%%%% July 2009 %%%%%%%%%%%%%%%%%%%%%%%%%
%
%% Version 1.0, July 2009.
%% Version 2.0, November 2010.
%

%%% Special commands used.
\def\R{\mathbb{R}}
\def\x{\mathbf{x}}
\def\X{\mathbf{X}}
\def\zero{\mathbf{0}}
\def\D{\mathcal{D}}
\def\d{\mathbf{d}}
\def\XX{\mathbb{X}}
\def\E{\operatorname{\mathbb{E}}}
\def\K{\operatorname{K}}
\def\crn{\operatorname{cr}}
\def\eu{\operatorname{Eu}}
\def\edisc{\operatorname{Edisc}}
\def\mxx{m(\mathbf{X})}
\def\mx{m(\mathbf{x})}

%
%%% BEGIN HERE %%%
%
\maketitle

\begin{abstract}
The expected value for the weighted crossing number of a randomly
weighted graph is studied. A variation of the Crossing Lemma for expectations is proved.
We focus on the case where the edge-weights
are independent random variables that are uniformly distributed on $[0,1]$.
\end{abstract}

% {\flushleft \bf Keywords:}
% graph, crossing number, weighted crossing number, crossing lemma.

%
%%%%% 1. INTRODUCTION %%%%%%%%%%%%%%%%%%%%%%%%%%%%%%%%%%%%%
%
\section{Introduction}\label{se:intro}

The {\it crossing number} of a graph is the minimum number of internal
intersections of edges in a drawing of the graph on the plane.
Computing the crossing number, even for complete graphs, is
a surprisingly challenging problem and an active area of
research \cite{RS09,SSV95,Vrto}.

The notion of the weighted crossing number, when the edges have weights
and each crossing counts as the product of the corresponding weights,
has been used in various situations, since it mimics the possibility of
having many edges in parallel.
In this paper we study the expected value of the weighted crossing
number of the complete graph $\K_n$ on $n$ vertices, where the weights
of edges are independent random variables.
We consider first the situation where the
weights are i.i.d.~variables with the uniform distribution
on $[0,1]$.  The first non-trivial case is $\K_5$;
we show through an involved calculation that
the expected value is $\frac{35921}{1108800}$.
We then briefly consider a simple discrete distribution
where the edges have value either $t$ or $u$, each with
probability $\frac{1}{2}$, and conclude that the expected
crossing number is not controlled by the first two moments
of the distribution on the edges.
Finally, we show that the expected crossing number of $\K_n$
retains the $\Theta(n^4)$ asymptotics of the usual crossing
number $\crn(\K_n)$ of complete graphs. This is proved by using
a similar recurrence as used for the usual crossing number of
complete graphs and, alternatively, by proving and applying
a variation of the Crossing Lemma for expectations.

%
%%%%% 2. PRELIMINARIES %%%%%%%%%%%%%%%%%%%%%%%%%%%%%%%%%%%%%
%
\section{Preliminaries}\label{se:prelim}

Given a graph $G=(V,E)$, we denote its crossing number by $\crn(G)$.
This is the minimum over all drawings of $G$ in the Euclidean plane $\R^2$
of the number of crossings of edges in the drawing. All drawings are
assumed to have simple polygonal arcs representing the edges of the graph,
and it is assumed that each pair of edges involves at most one intersection
of their representing arcs.
Here and in the remainder of the paper, we consider only internal
intersections of edges. Formally, a {\em crossing} in a drawing $\D$ is
an unordered pair $\{e,f\}$ of edges whose arcs in $\D$ intersect each other
internally. We let $\XX(\D)$ denote the set of all crossings and set
$\crn(\D) = |\XX(\D)|$.

Given non-negative weights $w: E\to \R_+$ on the edges of $G$, we define the
{\it crossing weight} of a drawing $\D$ of $G$ as:
$$
  \crn(\D,w) = \sum_{\{e,f\}\in \XX(\D)} w(e) w(f).
$$
%This is the natural measure if integer
%weights are used to represent multiple edges.
We define the {\it weighted crossing number} of a weighted graph $G$ as:
\begin{equation}\label{eq:cr}
   \crn(G,w) = \min_{\D} \crn(\D,w).
\end{equation}
For a fixed graph, the function $\crn(G,\cdot)$ is also called
the {\em crossing function} for $G$.
We take the domain of $\crn(G,\cdot)$ to be $\R^{E}_+$.
We remark that $\crn(G,0)=0$, $\crn(G,w) \ge 0$ and
$\crn(G,\cdot) \equiv \zero$ if and only if $\crn(G)=0$.
The function $\crn(G,\cdot)$ is piecewise quadratic in $w$,
and the chambers defined by these pieces correspond to
(groups of) optimal drawings for the contained weightings;
the forms in the chambers are neither convex nor concave.
If ${\mathbf 1}\in\R^{E}_+$ is the constant all-1 function,
then $\crn(G) = \crn(G,{\mathbf 1})$.

The crossing function of any $n$-vertex graph is just a specialization
of the crossing function $\crn(K_n,w)$ of the complete graph $K_n$,
where we put weight 0 for the non-edges in the graph. In this sense
the crossing functions of complete graphs contain information about
crossing numbers of all graphs. This universality property was the
main goal to introduce this notion in \cite{Mo1,Mo2} and to propose
its study.

Note that we allow the edges to be represented by any (polygonal)
line, they need not be straight lines.  The related question
of rectilinear crossing numbers is also interesting and
well-studied.  While rectilinear crossing number is in some
cases larger than the usual crossing number \cite{Guy},
they do not differ in the computations performed in this paper.
As in the unweighted case, minimal drawings can be obtained
without using double crossings (pairs of edges that cross
more than once).

%
%%%%% 3. K_5 %%%%%%%%%%%%%%%%%%%%%%%%%%%%%%%%%%%%%
%
\section{Computation of the expected crossing number}
\label{se:eu5}

We begin by considering the expected crossing number of the complete
graph $\K_n$ for some small values of $n$.  We take the weights
on the edges to be independently identically distributed
random variables, with uniform distributions on the interval $[0,1]$.
Let us denote the expected value of $\crn(\K_n,w)$ under this
distribution as $\eu(n)$.

For $n \le 4$, the graph can be drawn without crossings,
so $\eu(n)=0=\crn(\K_n)$.  For $n \ge 5$, we have $0 < \eu(n) < \crn(\K_n)$.
% In the following we compute $\eu(5)$.
In this section, we compute $\eu(5)$ directly from the definition
of expectation. Our somewhat cumbersome case analysis can also be viewed
as determination of the piecewise quadratic chambers for the crossing function
of $\K_5$.

We will number the edges of $\K_5$ as in Figure~\ref{fig:k5}.

\begin{figure}[hb]
\begin{center}
\includegraphics[height=50mm]{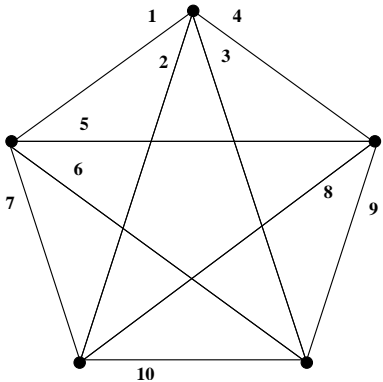}
\end{center}
\caption{Edge labelling of $K_5$} \label{fig:k5}
\end{figure}

We will denote the random weight assigned to the $i$th edge by
$\X_i$, $i=1,\dots,10$.  We note that $\crn(\K_5)=1$
and by symmetry, for any two non-adjacent edges, $\K_5$ can be drawn
so that those two edges are the single pair of crossing edges.
Hence:
\begin{align*}
\eu(5)=\E[\min( & \X_1 \X_8, \X_1 \X_9, \X_1 \X_{10}, \X_2 \X_5,
 \X_2 \X_6, \X_2 \X_9, \X_3 \X_5, \X_3 \X_7, \\
 & \X_3 \X_8, \X_4 \X_6, \X_4 \X_8, \X_4 \X_{10}, \X_5 \X_{10},
  \X_6 \X_8, \X_7 \X_9)]
\end{align*}
We abbreviate the quantity inside the expectation as $\mxx$.
%We remark that the adjacency graph of the edges of $\K_5$ is
%the complement of the Petersen graph, so this expectation can also be interpreted
%as the expectation of the minimum product of values of
%non-adjacent vertices in the Peterson graph given random
%$[0,1]$ variables on the vertices.

This is a problem in order statistics, see for instance \cite{DN03}.
The direct way to obtain $\eu(5)$ is to evaluate:
\begin{equation}\label{eq:integral}
\int_0^1 \int_0^1 \ldots \int_0^1 \mx \d x_1 \ldots \d x_9 \d x_{10}
\end{equation}
where $\mx$ is the function of $\x \in \R^{10}$ corresponding to the
random variables of $\mxx$.
To do this we break (\ref{eq:integral}) into $10!$ terms based on the
increasing order of the variables, i.e.~we compute (\ref{eq:integral}) via
the sum:
\begin{equation}\label{eq:pieces}
\sum_{\sigma \in S_{10}}
\int_0^1 \int_0^{x_{\sigma(10)}} \int_0^{x_{\sigma(9)}} \ldots
  \int_0^{x_{\sigma(2)}} \mx \d x_{\sigma(1)} \ldots
  \d x_{\sigma(9)} \d x_{\sigma(10)}
\end{equation}
Here the permutations $\sigma \in S_{10}$ index the
possible orderings of the random variables $\X$.
This sum has $10!$ terms, but they can be grouped into a manageable
number of cases.
To begin, we note that by reordering the vertices,
we can assume that $\X_1$ takes the smallest value, and,
using the labelling of Figure~\ref{fig:k5},
$\X_2 \le \X_3, \X_4, \X_5, \X_6, \X_7$
and $\X_3 \le \X_4$.
This corresponds to a labelling of $\K_5$ based on $\X$, breaking ties
arbitrarily. Actually, we may assume that the weights $\X_i$, $1\le i\le 10$,
are pairwise different, since the set on which an equality occurs is of measure
zero. Thus, each case with the above assumptions corresponds
to $120$ terms in (\ref{eq:pieces}).

With these assumptions, the minimum of the 15 pairs of random
variables in $\mxx$ must be attained at
one of $\X_1 \X_8, \X_1 \X_9, \X_1 \X_{10}, \X_2 \X_5, \X_2 \X_6, \X_3 \X_7$
since $\X_1 \X_9 \le \X_2 \X_9, \X_7 \X_9$;
$\X_2 \X_5 \le \X_3 \X_5;$ et cetera.
We note that these six terms are symmetric in the variables
$\X_8, \X_9, \X_{10}$, and also in $\X_5, \X_6$.
Thus we will also take $\X_8 =\min(\X_8,\X_9,\X_{10})$
and $\X_5=\min(\X_5,\X_6)$, and treat the remaining cases
by symmetry.  Combined with our assumptions on $\X_1, \X_2$
and $\X_3$ we break the $10!$ terms of (\ref{eq:pieces})
into groups of 720 terms based on symmetry;
this leaves us with 5040 terms up to these symmetries.
It also allows us to simplify our integrand
further to $\min(\X_1 \X_8, \X_2 \X_5, \X_3 \X_7)$.

We now divide into cases based on the relative orderings of some
of the remaining variables.  We remark that, depending on
the order of the variables, the integrand may simplify further --
for instance if the two smallest variables are $\X_1$ and $\X_8$,
the minimum of the three terms will always be $\X_1 \X_8$.
We organize the cases by how the integrand simplifies.

\begin{flushleft}
\underline{Case 1}: Orderings which ensure
$\X_1 \X_8 = \min(\X_1 \X_8, \X_2 \X_5, \X_3 \X_7)$.
\end{flushleft}

In these cases, the computation is relatively simple:
the integral depends only on which position $\X_8$
occupies in the order of the $\X_i$'s.
It can be anywhere from the second to fifth smallest.
Suppose it is the second smallest, i.e.~that the order of the
variables is:
$$\X_1 \le \X_8 \le \X_{i_3} \le \X_{i_4} \le \X_{i_5} \le \X_{i_6}
  \le \X_{i_7} \le \X_{i_8} \le \X_{i_9} \le \X_{i_{10}}.$$
Then we compute:
$$
\int_0^1 \int_0^{x_{i_{10}}} \int_0^{x_{i_9}}
 % \int_0^{x_{i_8}} \int_0^{x_{i_7}} \int_0^{x_{i_6}} \int_0^{x_{i_5}}
 \ldots \int_0^{x_{i_4}}
 \int_0^{x_{i_3}} \int_0^{x_8} x_1 x_8 \d x_1 \d x_8 \d x_{i_3}
 % \d x_{i_4} \d x_{i_5} \d x_{i_6} \d x_{i_7}
  \ldots \d x_{i_8} \d x_{i_9} \d x_{i_{10}}
$$
$$
 = \int_0^1 \int_0^{x_{i_{10}}} \int_0^{x_{i_9}}
 % \int_0^{x_{i_8}} \int_0^{x_{i_7}} \int_0^{x_{i_6}} \int_0^{x_{i_5}}
  \ldots \int_0^{x_{i_4}}
 \int_0^{x_{i_3}} \frac{x_8^3}{2} \d x_8 \d x_{i_3}
 % \d x_{i_4} \d x_{i_5} \d x_{i_6} \d x_{i_7}
 \ldots \d x_{i_8} \d x_{i_9} \d x_{i_{10}}
$$
$$
 = \int_0^1 \int_0^{x_{i_{10}}} \int_0^{x_{i_9}}
 % \int_0^{x_{i_8}} \int_0^{x_{i_7}} \int_0^{x_{i_6}} \int_0^{x_{i_5}}
 \ldots \int_0^{x_{i_4}}
 \frac{x_{i_3}^4}{2 \cdot 4} \d x_{i_3}
 % \d x_{i_4} \d x_{i_5} \d x_{i_6} \d x_{i_7}
 \ldots \d x_{i_8} \d x_{i_9} \d x_{i_{10}}
$$
\begin{center}$\ldots$\end{center}
\vspace{2mm}
$$ = \int_0^1 \frac{x_{i_{10}}^{11}}{2 \cdot 4 \cdot 5 \cdot 6 \cdot 7
  \cdot 8 \cdot 9 \cdot 10 \cdot 11} \d x_{i_{10}}= \frac{3}{12!}
$$

A similar calculation shows that if $\X_8$ is $i^{th}$ smallest
variable, the integral for a fixed ordering of the remaining
variables will be $\frac{i+1}{12!}$.

Now observe that there are $\frac{8!}{24}=1680$ ways
of ordering the variables
with $\X_1$ as the smallest variable,
$\X_8$ as the second smallest,
$\X_2 \le \X_3, \X_4, \X_5, \X_6, \X_7$,
$\X_3 \le \X_4$, and $\X_5 \le \X_6$.
We remark that our symmetry assumptions guarantee that either
$\X_2$ or $\X_8$ is the second smallest variable, so in the
remainder of the analysis $\X_2$ will always be the second smallest
variable.

Thus if $\X_8$ is the third smallest variable, we have
fixed the order of the first 3 variables, and the remaining
variables can be ordered in $7!/4=1260$ ways, accounting for the
facts that $\X_3 \le \X_4$ and $\X_5 \le \X_6$.

If $\X_8$ is the fourth smallest variable, we have two
possible choices for the third smallest: $\X_3$ and $\X_7$.
In the former case we have $6!/2=360$ possible ordering of
the remaining variables (accounting for $\X_5 \le \X_6$),
while for the latter case we have $6!/4=180$ possible
orderings.

Finally, under the assumptions of Case 1,
$\X_8$ can be the fifth smallest variable only if
the third and fourth variables are $\X_3$ and $\X_4$,
respectively.  There are $5!/2$ orderings of the remaining
variables compatible with this.
We remark that in this case, we can never have $\X_5 \le \X_8$,
or $\X_3, \X_7 \le \X_8$ since then it may be the case that
$\X_1 \X_8$ is not minimal, depending on the values chosen.

This already covers the majority of the cases, 3540 of the 5040.
Thus the terms in (\ref{eq:pieces}) corresponding to these
orderings of the variables have total weight per symmetry class of:
$$1680 \cdot \frac{3}{12!} +
1260 \cdot \frac{4}{12!} +
540 \cdot \frac{5}{12!} +
60 \cdot \frac{6}{12!} = \frac{13140}{12!}.
$$

\begin{flushleft}
\underline{Case 2}: Orderings which ensure
$\min(\X_1 \X_8, \X_2 \X_5, \X_3 \X_7)$ is either
$\X_1 \X_8$ or $\X_2 \X_5$.
\end{flushleft}

In these cases, $\X_2$ and $\X_5$ are between $\X_1$ and $\X_8$.
However, $\X_3$ and $\X_7$ are not both between $\X_2$ and $\X_5$.
The integrand will be $\mxx=\min(\X_1 \X_8, \X_2 \X_5)$,
and the two smallest variables are $\X_1$ and $\X_2$.
We break into subcases based on the positions of $\X_5$ and $\X_8$.
Only the simplest case is described in detail.

\begin{flushleft}
\underline{Subcase 2i}: The four smallest variables are
$\X_1, \X_2, \X_5$ and $\X_8$.
Then we need to evaluate:
\end{flushleft}
$$
\int_0^1 \int_0^{x_{i_{10}}}  \ldots
 \int_0^{x_{i_6}} \int_0^{x_{i_5}} \int_0^{x_8}
 \int_0^{x_5} \int_0^{x_2} \min(x_1 x_8, x_2 x_5) \d x_1 \d x_2 \d x_5
 \d x_8 \d x_{i_5} \ldots \d x_{i_9}
 \d x_{i_{10}}
$$
$$
= \int_0^1 \int_0^{x_{i_{10}}}  \ldots
 \int_0^{x_{i_6}} \int_0^{x_{i_5}} \int_0^{x_8}
 \int_0^{x_5} \int_0^{\frac{x_2 x_5}{x_8}} x_1 x_8 \d x_1 \d x_2 \d x_5
 \d x_8 \d x_{i_5} \ldots \d x_{i_9} \d x_{i_{10}}
$$
$$
\qquad + \int_0^1 \int_0^{x_{i_{10}}}  \ldots
 \int_0^{x_{i_6}} \int_0^{x_{i_5}} \int_0^{x_8}
 \int_0^{x_5} \int_{\frac{x_2 x_5}{x_8}}^{x_2} x_2 x_5 \d x_1 \d x_2 \d x_5
 \d x_8 \d x_{i_5} \ldots \d x_{i_9} \d x_{i_{10}}
$$
$$
= \int_0^1 \int_0^{x_{i_{10}}}  \ldots
 \int_0^{x_{i_6}} \int_0^{x_{i_5}} \int_0^{x_8}
 \int_0^{x_5} \left( \frac{x_2^2 x_5^2}{2 x_8}
  + x_2^2 x_5 - \frac{x_2^2 x_5^2}{x_8} \right) \d x_2 \d x_5
 \d x_8 \d x_{i_5} \ldots \d x_{i_9} \d x_{i_{10}}
$$
$$
= \int_0^1 \int_0^{x_{i_{10}}}  \ldots
 \int_0^{x_{i_6}} \int_0^{x_{i_5}} \int_0^{x_8}
  \left( \frac{x_5^4}{3} - \frac{x_5^5}{6 x_8} \right) \d x_5
 \d x_8 \d x_{i_5} \ldots \d x_{i_9} \d x_{i_{10}}
$$
$$
= \int_0^1 \int_0^{x_{i_{10}}}  \ldots
 \int_0^{x_{i_6}} \int_0^{x_{i_5}}
  \left( \frac{x_8^5}{15} - \frac{x_8^5}{36} \right)
 \d x_8 \d x_{i_5} \ldots \d x_{i_9} \d x_{i_{10}}
$$
$$
= \int_0^1 \int_0^{x_{i_{10}}}  \ldots
 \int_0^{x_{i_6}}
  \frac{7 x_{i_5}^6}{180}\,
 \d x_{i_5} \ldots \d x_{i_9} \d x_{i_{10}} =
 \int_0^1 \frac{7 \cdot 5! \, x_{i_{10}}^{11}}{180 \cdot 11!} \d x_{i_{10}}
 = \frac{14}{3 \cdot 12!}.
$$
The number of orderings of the variables in Subcase 2i
up to symmetries is $\frac{6!}{2}=360$, since we require $\X_3 \le \X_4$.

The integrals in the remaining cases are essentially similar, so we will
simply list the initial sequence of integrands and then compute the
number of orderings of the variables corresponding to each case.

\begin{flushleft}
\underline{Subcase 2ii}: The five smallest variables are
$\X_1, \X_2, \X_5, \X_j$ and $\X_8$.  This produces the
following integrands:
\end{flushleft}
$$
\min(x_1 x_8, x_2 x_5);
x_2^2 x_5 - \frac{x_2^2 x_5^2}{2 x_8};
\frac{x_5^4}{3}-\frac{x_5^5}{6 x_8};
\frac{x_j^5}{15}-\frac{x_j^6}{36 x_8};
\frac{x_8^6}{140}; \frac{x_{i_6}^7}{140 \cdot 7}; \ldots
\frac{6! x_{i_{10}}}{140 \cdot 11!}; \frac{36}{7 \cdot 12!}.
$$
There are three possibilities for $j$: 3, 6 and 7.
When $j=3$ we have $5!$ cases, and when $j=6$ or $j=7$
we have $\frac{5!}{2}$ cases as we need to account for
the fact that $\X_3 \le \X_4$ in the remaining variables.
This is the total of 240 orderings.

\begin{flushleft}
\underline{Subcase 2iii}: The six smallest variables are
$\X_1, \X_2, \X_5, \X_j, \X_k$ and $\X_8$.
The integrands remain as in the previous case up to
$\frac{x_j^5}{15}-\frac{x_j^6}{36 x_8}$.
The next integrands will be
$\frac{x_k^6}{90}-\frac{x_k^7}{36 \cdot 7 x_8}$
and $\frac{11 x_8^7}{7 \cdot 1440}$, the remaining simple
integrations give
$\frac{11}{2 \cdot 12!}$.
\end{flushleft}

There are seven possibilities for $(j,k)$ in this case:
$(3,4), (3,6), (3,7), (6,3), (7,3), (6,7)$ and $(7,6)$.
The first five of these are each associated to $4!$ orderings
of the remaining variables, while the last two are each
associated to $\frac{4!}{2}$.   This is a total of 144 orderings.

\begin{flushleft}
\underline{Subcase 2iv}: The seven smallest variables are
$\X_1, \X_2, \X_5, \X_j, \X_k, \X_l$ and $\X_8$.
We proceed from the integrand
$\frac{x_k^6}{90}-\frac{x_k^7}{36 \cdot 7 x_8}$
to $\frac{x_l^7}{7 \cdot 90}
 -\frac{x_l^8}{7 \cdot 8 \cdot 36 x_8}$,
and then to $\frac{13 x_8^8}{7 \cdot 8 \cdot 9 \cdot 180}$.
Continuing to the end, the integral is
$\frac{52}{9 \cdot 12!}$.
\end{flushleft}

There are twelve possibilities for $(j,k,l)$ arising from
choosing three of $3,4,6,7$ and requiring 4 to be preceded
by 3.  Each of these has $3!$ orderings of the remaining 3
variables, for a total of 72 orderings.

\begin{flushleft}
\underline{Subcase 2v}: The eight smallest variables are
$\X_1, \X_2, \X_5, \X_j, \X_k, \X_l, \X_m$ and $\X_8$.
We proceed from the integrand
$\frac{x_l^7}{7 \cdot 90} -\frac{x_l^8}{7 \cdot 8 \cdot 36 x_8}$
to
$\frac{x_m^8}{7 \cdot 8 \cdot 90} -
 \frac{x_m^9}{7 \cdot 8 \cdot 9 \cdot 36 x_8}$
and then to $\frac{x_8^9}{7 \cdot 8 \cdot 9 \cdot 120}$.
Continuing to the end, the integral is
$\frac{6}{12!}$.
\end{flushleft}

Again there are 12 possibilities, as $(j,k,l,m)$ are chosen
from $3,4,6,7$ with 3 preceding 4.  There are 2 ways of
arranging the remaining two variables, for a total of 24
orderings.

\begin{flushleft}
\underline{Subcase 2vi}: The five smallest variables are
$\X_1, \X_2, \X_j, \X_5$ and $\X_8$.  The sequence of
integrands that we see is then:
\end{flushleft}
$$
\min(x_1 x_8, x_2 x_5);
x_2^2 x_5 - \frac{x_2^2 x_5^2}{2 x_8};
\frac{x_j^3 x_5}{3}-\frac{x_j^3 x_5}{6 x_8};
\frac{x_5^5}{12}-\frac{x_5^6}{24 x_8};
\frac{x_8^6}{126}; \frac{x_{i_6}^7}{126 \cdot 7}; \ldots
\frac{6! x_{i_{10}}}{126 \cdot 11!}; \frac{40}{7 \cdot 12!}.
$$

In this case, $j$ must be either 3 or 7.  If it is 3,
there are $5!$ ways of ordering the remaining variables,
and if it is 7 there are $\frac{5!}{2}$ ways of ordering
the remaining variables, for a total of 180 orderings.

\begin{flushleft}
\underline{Subcase 2vii}: The six smallest variables are
$\X_1, \X_2, \X_j, \X_5, \X_k$ and $\X_8$.
We then see the same integrands through
$\frac{x_5^5}{12}-\frac{x_5^6}{24 x_8}$
followed by $\frac{x_k^6}{72}-\frac{x_k^7}{7 \cdot 24 x_8}$,
$\frac{x_8^7}{7 \cdot 72}-\frac{x_8^7}{7 \cdot 8 \cdot 24}$.
Following the remaining routine integrations, we get
$\frac{25}{4 \cdot 12!}$.
\end{flushleft}

The possibilities for $(j,k)$ are similar to those of Subcase 2iii,
but 6 cannot be used in the first position.  This leaves
$(3,4), (3,6), (3,7), (7,3)$ and $ (7,6)$.
The first four cases correspond to $4!$ orderings of the remaining
variables, while $(7,6)$ corresponds to $\frac{4!}{2}$ as
3 must precede 4.  This gives a total of 108 orderings.

\begin{flushleft}
\underline{Subcase 2viii}: The seven smallest variables are
$\X_1, \X_2, \X_j, \X_5, \X_k, \X_l$ and $\X_8$.
We proceed as in the previous subcase through
$\frac{x_k^6}{72}-\frac{x_k^7}{7 \cdot 24 x_8}$, then to
$\frac{x_l^7}{7 \cdot 72}-\frac{x_l^8}{7 \cdot 8 \cdot 24 x_8}$,
and $\frac{x_8^8}{7 \cdot 8 \cdot 72}-\frac{x_8^8}{7 \cdot 8 \cdot 9 \cdot 24}$.
The remaining integrations bring us to
$\frac{20}{3 \cdot 12!}$.
\end{flushleft}

In this subcase we have $(j,k,l)$ chosen from $3,4,6,7$ with
the conditions that 4 must be preceded by 3, and 6 may not
appear in the first position.  This second condition removes
3 of the 12 orderings as compared to Subcase 2iv, leaving us
with 9.  There are always $3!$ orderings of the remaining
variables, giving a total of 54 orderings for this case.

\begin{flushleft}
\underline{Subcase 2ix}: The eight smallest variables are
$\X_1, \X_2, \X_j, \X_5, \X_k, \X_l, \X_m$ and $\X_8$.
We proceed as in the previous subcase through
$\frac{x_l^7}{7 \cdot 72}-\frac{x_l^8}{7 \cdot 8 \cdot 24 x_8}$.
The next integrand is
$\frac{x_m^8}{7 \cdot 8 \cdot 72}-
  \frac{x_m^9}{7 \cdot 8 \cdot 9 \cdot 24 x_8}$, followed by
$\frac{7 x_8^9}{24 \cdot 30 \cdot 7 \cdot 8 \cdot 9}$,
and eventually $\frac{7}{12!}$.
\end{flushleft}

There are 9 possible choices for $(j,k,l,m)$ since we have
the same conditions as in the previous subcase, with the
remaining number assigned to $m$.  There are two orders
for the remaining two variables, giving a total of 18
orderings.

\begin{flushleft}
\underline{Subcase 2x}: The six smallest variables are
$\X_1, \X_2, \X_j, \X_k, \X_5$ and $\X_8$.
We proceed as in the previous subcase through the integrand
$\frac{x_j^3 x_5}{3}-\frac{x_j^3 x_5^2}{6 x_8}$
Continuing, we see integrands
$\frac{x_k^4 x_5}{12}-\frac{x_k^4 x_5^2}{24 x_8}$,
$\frac{x_5^6}{60}-\frac{x_5^7}{120 x_8}$, and
$\frac{x_8^7}{7 \cdot 60}-\frac{x_8^7}{16 \cdot 60}$
on our way to $\frac{27}{4 \cdot 12!}$.
\end{flushleft}

In fact this case requires $j=3$ and $k=4$, since we
can't have either $\X_6$ or both of $\X_3$ and $\X_7$
precede $\X_5$.
There are $4!=24$ ways of ordering the remaining variables.

\begin{flushleft}
\underline{Subcase 2xi}: The seven smallest variables are
$\X_1, \X_2, \X_3, \X_4, \X_5, \X_l$ and $\X_8$.
This matches the previous subcase through
$\frac{x_5^6}{60}-\frac{x_5^7}{120 x_8}$;
the next two integrands are
$\frac{x_l^7}{7 \cdot 60} - \frac{x_l^8}{8 \cdot 120 x_8}$ and
$\frac{x_8^8}{7 \cdot 8 \cdot 60}-\frac{x_8^8}{8 \cdot 9 \cdot 120}$.
Continuing we arrive at $\frac{22}{3 \cdot 12!}$.
\end{flushleft}

We must have $l=6$ or $l=7$.  There are $3!$ ways of ordering
the remaining variables for a total of 12 orderings.

\begin{flushleft}
\underline{Subcase 2xii}: The eight smallest variables are
$\X_1, \X_2, \X_3, \X_4, \X_5, \X_l, \X_m$ and $\X_8$.
This matches the previous subcase through
$\frac{x_l^7}{7 \cdot 60} - \frac{x_l^8}{8 \cdot 120 x_8}$.
Next we have
$\frac{x_m^8}{7 \cdot 8 \cdot 60}-
 \frac{x_m^9}{8 \cdot 9 \cdot 120 x_8}$.
and
$\frac{x_8^9}{7 \cdot 8 \cdot 9 \cdot 60}-
 \frac{x_8^9}{8 \cdot 9 \cdot 10 \cdot 120}$.
Continuing we arrive at $\frac{39}{5 \cdot 12!}$.
\end{flushleft}

We must have $(l,m)$ equal to $(6,7)$ or $(7,6)$,
and there are 2 ways of ordering the remaining 2
variables, for a total of 4 orderings.

This completes Case 2, which contains 1240 possible orderings
up to the symmetries.
The terms in (\ref{eq:pieces}) corresponding to these
orderings of the variables have total weight per symmetry class of:
$360 \cdot \frac{14}{3 \cdot 12!} + \ldots +
4 \cdot \frac{39}{5 \cdot 12!} = \frac{235797}{35 \cdot 12!}$.

\begin{flushleft}
\underline{Case 3}: Orderings which ensure
$\min(\X_1 \X_8, \X_2 \X_5, \X_3 \X_7)$ is either
$\X_1 \X_8$ or $\X_3 \X_7$.
%$\X_2 \X_5$.
\end{flushleft}

Since $\X_2$ is the second smallest variable, these will
occur only when $\X_3, \X_7 \le \X_8$, but $\X_8 \le \X_5$.
Only the simplest case is described in detail.

\begin{flushleft}
\underline{Subcase 3i}: The five smallest variables are
$\X_1, \X_2, \X_3, \X_7$ and $\X_8$  or
$\X_1, \X_2, \X_7, \X_3$ and $\X_8$.
\end{flushleft}

We proceed to evaluate:
$$
\int_0^1 \int_0^{x_{i_{10}}}  \ldots
 \int_0^{x_{i_6}} \int_0^{x_8} \int_0^{x_{7}}
 \int_0^{x_3} \int_0^{x_2} \min(x_1 x_8, x_3 x_{7}) \d x_1 \d x_2 \d x_3
 \d x_7 \d x_8 \ldots \d x_{i_9} \d x_{i_{10}}
% \int_0^{x_{i_6}} \int_0^{x_{i_5}} \int_0^{x_{7}}
% \int_0^{x_3} \int_0^{x_2} \min(x_1 x_8, x_3 x_{7}) \d x_1 \d x_2 \d x_5
% \d x_8 \d x_{i_5} \ldots \d x_{i_9} \d x_{i_{10}}
$$
The inner integral is piecewise linear in $x_1$, with a single
break point at $x_1 = \frac{x_3 x_{7}}{x_8}$, however
$\frac{x_3 x_{7}}{x_8}$ may or may not be greater than $x_2$.
We decompose the inner integral as:
$$\int_0^{\min(x_2,\frac{x_3 x_{7}}{x_8})} x_1 x_8 \d x_1 +
 \int_{\min(x_2,\frac{x_3 x_{7}}{x_8})}^{x_2} x_3 x_{7} \d x_1$$
Evaluating this integral leaves us with the new inner integral:
$$\int_0^{x_3} \left(\frac{1}{2} \min(x_2, \frac{x_3 x_{7}}{x_8})^2 x_8 +
 x_2 x_3 x_{7} - \min(x_2, \frac{x_3 x_{7}}{x_8}) x_3 x_{7}\right) \d x_2$$
This again needs to be split, this time with breakpoint at
$x_2 = \frac{x_3 x_{7}}{x_8}$:
$$\int_0^{\frac{x_3 x_{7}}{x_8}} \left(\frac{x_2^2 x_8}{2} + x_2 x_3 x_{7}
 - x_2 x_3 x_{7}\right)  \d x_2 +
\int_{\frac{x_3 x_{7}}{x_8}}^{x_3} \left(\frac{x_3^2 x_{7}^2}{2 x_8}
 + x_2 x_3 x_{7} - \frac{x_3^2 x_{7}^2}{x_8}\right) \d x_2
$$
Happily, we see some cancelation of terms, both before evaluating
the integral and after.  This yields:
$$\int_0^{x_7} \left(\frac{x_3^3 x_{7}^3}{6 x_8^2} +
 \frac{x_3^3 x_{7}}{2} - \frac{x_3^3 x_{7}^2}{2 x_8}\right) \d x_3
$$
We proceed through the following integrands:
$$\frac{x_{7}^7}{24 x_8^2} + \frac{x_{7}^5}{8} - \frac{x_{7}^6}{8 x_8};
\ \frac{11 x_8^6}{1344}; \ \ldots; \ \frac{165}{28 \cdot 12!}
$$

Accounting for the fact that $\X_5 \le \X_6$, there are
$\frac{5!}{2}$ orderings of the remaining variables.
With the two orderings of $\X_3$ and $\X_7$ (which do not
affect the computation of the integral),
we have the total of 120 orderings corresponding to this subcase.

\begin{flushleft}
\underline{Subcase 3ii}: The six smallest variables are
$\X_1, \X_2, \X_3, \X_{7}, \X_4$ and $\X_8$, with $\X_3$
and $\X_7$ possibly switched.
\end{flushleft}

The integrands then remain identical through
$\frac{x_{7}^7}{24 x_8^2} + \frac{x_{7}^5}{8} - \frac{x_{7}^6}{8 x_8}$.
These are followed by
%$\frac{x_4^8}{8 \cdot 24 x_8^2} + \frac{x_4^5}{6 \cdot 8}
%  - \frac{x_4^6}{7 \cdot 8 x_8}$ and
%$\frac{x_8^7}{756}$.  Subsequent integrations yield
$\frac{x_4^8}{8 \cdot 24 x_8^2} + \frac{x_4^6}{6 \cdot 8}
  - \frac{x_4^7}{7 \cdot 8 x_8}$ and
$\frac{x_8^7}{756}$.  Subsequent integrations yield
$\frac{20}{3 \cdot 12!}$.
There are $\frac{4!}{2}$ orderings of the remaining variables,
and $\X_3$ and $\X_7$ can be switched, giving a total of 24
orderings.

\begin{flushleft}
\underline{Subcase 3iii}: The six smallest variables are
$\X_1, \X_2, \X_3, \X_4, \X_{7}$ and $\X_8$.
\end{flushleft}

The integrands are identical to Subcase 3i until
$\frac{x_3^3 x_{7}^3}{6 x_8^2} +
 \frac{x_3^3 x_{7}}{2} - \frac{x_3^3 x_{7}^2}{2 x_8}$.
We proceed to:
$$\frac{x_4^4 x_{7}^3}{24 x_8^2} + \frac{x_4^4 x_{7}}{8}
 - \frac{x_4^4 x_{7}^2}{8 x_8}; \quad
\frac{x_{7}^8}{120 x_8^2} + \frac{x_{7}^6}{40} - \frac{x_{7}^7}{40 x_8};
 \quad
\frac{83 x_8^7}{60480}; \quad \ldots \quad \frac{83}{12 \cdot 12!}
$$
There are $\frac{4!}{2} = 12$ orderings of the remaining
variables.

This completes Case 3, which comprises 156 orderings of the
variables.
The terms in (\ref{eq:pieces}) corresponding to these
orderings of the variables have total weight per symmetry class of
$\frac{6651}{7 \cdot 12!}$.

\begin{flushleft}
\underline{Case 4}: Orderings in which
$\min(\X_1 \X_8, \X_2 \X_5, \X_3 \X_7)$ can be
attained at all 3 terms.
\end{flushleft}

This comprises a small number of orderings that feature a
messy inner integral.  We note that in all these cases
the two smallest variables are $\X_1$ and $\X_2$, while
the third smallest variable is either $\X_3$ or $\X_7$.
Our integrand is symmetric in $\X_3$ and $\X_7$, so we
will do the computation only with $\X_3$ as the smaller
of the two variables.  We will proceed to evaluate the
three innermost integrals before breaking into subcases,
assuming that the fourth smallest variable is $\X_j$:

$$ \int_0^{x_j}
 \int_0^{x_3} \int_0^{x_2} \min(x_1 x_8, x_2 x_5, x_3 x_7) \d x_1 \d x_2 \d x_3
$$
$$
= \int_0^{x_j} \int_0^{x_3}
   \int_0^{\frac{\min(x_2 x_5,x_3 x_7)}{x_8}} x_1 x_8
  \d x_1 % \d x_2 \d x_3
 + % \int_0^{x_j} \int_0^{x_3}
   \int_{\frac{\min(x_2 x_5,x_3 x_7)}{x_8}}^{x_2}
   \min(x_2 x_5, x_3 x_7)
  \d x_1 \d x_2 \d x_3
$$
$$
= \int_0^{x_j} \int_0^{x_3}
   \left( \frac{\min(x_2 x_5,x_3 x_7)^2}{2 x_8}
 +   x_2 \min(x_2 x_5,x_3 x_7)
 -   \frac{\min(x_2 x_5,x_3 x_7)^2}{x_8} \right)
  \d x_2 \d x_3
$$
$$
= \int_0^{x_j} \int_0^{x_3}
   \left( x_2 \min(x_2 x_5,x_3 x_7)
   - \frac{\min(x_2 x_5,x_3 x_7)^2}{2 x_8}\right)
  \d x_2 \d x_3
$$
$$
= \int_0^{x_j} \int_0^{\frac{x_3 x_7}{x_5}}
   \left(x_2^2 x_5 - \frac{x_2^2 x_5^2}{2 x_8} \right) \d x_2 \d x_3
  + \int_{\frac{x_3 x_7}{x_5}}^{x_3}
   \left( x_2 x_3 x_7 - \frac{x_3^2 x_7^2}{2 x_8} \right) \d x_2
\d x_3
$$
$$
= \int_0^{x_j} \left(
   \frac{x_3^3 x_7^3 x_5}{3 x_5^3} - \frac{x_3^3 x_7^3 x_5^2}{6 x_8 x_5^3}
  + \frac{x_3^3 x_7}{2} - \frac{x_3^3 x_7^2}{2 x_8}
  - \frac{x_3^3 x_7^3}{2 x_5^2} + \frac{x_3^3 x_7^3}{2 x_8 x_5} \right)
\d x_3
$$
$$
= \int_0^{x_j} x_3^3
  \left[
  \frac{x_7}{2} - \frac{x_7^2}{2 x_8}
  - \frac{x_7^3}{6 x_5^2} + \frac{x_7^3}{3 x_8 x_5}
  \right]
\d x_3
= \frac{x_j^4}{4}
  \left[
  \frac{x_7}{2} - \frac{x_7^2}{2 x_8}
  - \frac{x_7^3}{6 x_5^2} + \frac{x_7^3}{3 x_8 x_5}
  \right]
$$

\medskip

We now proceed to the subcases which are based on the
ordering of the remaining variables.

\begin{flushleft}
\underline{Subcase 4i}: The six smallest variables are
$\X_1, \X_2, \X_3, \X_7, \X_5$ and $\X_8$.
Then $\X_j=\X_7$ and we need to evaluate:
\end{flushleft}
$$
\int_0^1 \int_0^{x_{i_{10}}}
 \int_0^{x_{i_9}} \int_0^{x_{i_8}}
 \int_0^{x_{i_7}} \int_0^{x_8} \int_0^{x_5}
  \left( \frac{x_7^5}{8} - \frac{x_7^6}{8 x_8}
  - \frac{x_7^7}{24 x_5^2} + \frac{x_7^7}{12 x_8 x_5} \right)
  \d x_7 \d x_5 \d x_8 \d x_{i_7} \d x_{i_8} \d x_{i_9}
 \d x_{i_{10}}
$$
% $$
% = \frac{1}{4} \int_0^1 \int_0^{x_{i_{10}}}
%  \int_0^{x_{i_9}} \int_0^{x_{i_8}} \int_0^{x_{i_7}}
%  \int_0^{x_8}
%   \frac{x_5^6}{12} - \frac{x_5^7}{14 x_8}
%   - \frac{x_5^8}{48 x_5^2} + \frac{x_5^8}{24 x_8 x_5}
%   \d x_5 \d x_8 \d x_{i_7} \d x_{i_8} \d x_{i_9}
%  \d x_{i_{10}}
% $$
$$
= \frac{1}{32} \int_0^1 \int_0^{x_{i_{10}}}
 \int_0^{x_{i_9}} \int_0^{x_{i_8}} \int_0^{x_{i_7}}
 \int_0^{x_8}
  \left( \frac{x_5^6}{2} - \frac{5 x_5^7}{21 x_8} \right)
  \d x_5 \d x_8 \d x_{i_7} \d x_{i_8} \d x_{i_9}
 \d x_{i_{10}}
$$
% $$
% = \frac{1}{32} \int_0^1 \int_0^{x_{i_{10}}}
%  \int_0^{x_{i_9}} \int_0^{x_{i_8}} \int_0^{x_{i_7}}
%   \frac{x_8^7}{14} - \frac{5 x_8^8}{168 x_8}
%   \d x_8 \d x_{i_7} \d x_{i_8} \d x_{i_9}
%  \d x_{i_{10}}
% $$
$$
= \frac{1}{768} \int_0^1 \int_0^{x_{i_{10}}}
 \int_0^{x_{i_9}} \int_0^{x_{i_8}} \int_0^{x_{i_7}}
  {x_8^7}
  \d x_8 \d x_{i_7} \d x_{i_8} \d x_{i_9}
 \d x_{i_{10}}
  = \frac{7!}{768 \cdot 12!} = \frac{105}{16 \cdot 12!}.
$$

\medskip

As noted previously, there is a second ordering corresponding
to this subcase, in which $\X_3$ and $\X_7$ are reversed,
and there are $4!$ orderings of the remaining variables,
giving us 48 orderings in this case.

\begin{flushleft}
\underline{Subcase 4ii}: The seven smallest variables are
$\X_1, \X_2, \X_3, \X_7, \X_5$, $\X_k$ and $\X_8$.
\end{flushleft}

This calculation is quite similar to the previous one until
it reaches the integral with integrand
$\frac{1}{32} [\frac{x_5^6}{2} - \frac{5 x_5^7}{21 x_8}]$.
Subsequent integrands are:
$$\frac{1}{32} \left[\frac{x_k^7}{14} - \frac{5 x_k^8}{168 x_8} \right];
\frac{1}{32} \left[\frac{x_8^8}{112} - \frac{5 x_8^9}{1512 x_8} \right];
\text{ and }
\frac{1}{1792} \left[\frac{17 x_8^8}{54}\right].
$$
The remaining integrations bring us to $\frac{85}{12 \cdot 12!}$.

% $$
% \frac{1}{32} \int_0^1 \int_0^{x_{i_{10}}}
%  \int_0^{x_{i_9}} \int_0^{x_{i_8}}
%  \int_0^{x_8}
%   \frac{x_k^7}{14} - \frac{5 x_k^8}{168 x_8}
%   \d x_k \d x_8 \d x_{i_8} \d x_{i_9}
%  \d x_{i_{10}}
% $$
% $$
% = \frac{1}{32} \int_0^1 \int_0^{x_{i_{10}}}
%  \int_0^{x_{i_9}} \int_0^{x_{i_8}}
%   \frac{x_8^8}{112} - \frac{5 x_8^9}{1512 x_8}
%   \d x_8 \d x_{i_8} \d x_{i_9}
%  \d x_{i_{10}}
% $$
% $$
% = \frac{1}{1792} \int_0^1 \int_0^{x_{i_{10}}}
%  \int_0^{x_{i_9}} \int_0^{x_{i_8}}
%   \frac{17 x_8^8}{54}
%   \d x_8 \d x_{i_8} \d x_{i_9}
%  \d x_{i_{10}} = \frac{17 \cdot 8!}{1792 \cdot 54 \cdot 12!}
%   = \frac{85}{12 \cdot 12!}.
% $$
There are $3!$ orderings of the remaining variables,
$k$ may be either 4 or 6, and $\X_3$ and $\X_7$ can again be
reversed, giving us a total of 24 orderings in this subcase.

\begin{flushleft}
\underline{Subcase 4iii}: The seven smallest variables are
$\X_1, \X_2, \X_3, \X_7, \X_4$, $\X_5$ and $\X_8$.
\end{flushleft}
The first integration is similar to the first integration
in Subcase 4i, and we proceed from there, via the following
integrands:
$$
\frac{1}{4} \left[ \frac{x_4^6}{12} - \frac{x_4^7}{14 x_8}
   - \frac{x_4^8}{48 x_5^2} + \frac{x_4^8}{24 x_8 x_5}\right];
\frac{1}{12096} \left[ {29 x_5^7} - \frac{13 x_5^8}{x_8} \right];
\frac{157 x_8^8}{870912}.
$$
The remaining integrations bring us to $\frac{785}{108 \cdot 12!}$.
% $$
% \frac{1}{4} \int_0^1 \int_0^{x_{i_{10}}}
%  \int_0^{x_{i_9}} \int_0^{x_{i_8}}
%  \int_0^{x_8} \int_0^{x_5}
%   \frac{x_4^6}{12} - \frac{x_4^7}{14 x_8}
%   - \frac{x_4^8}{48 x_5^2} + \frac{x_4^8}{24 x_8 x_5}
%   \d x_4 \d x_5 \d x_8 \d x_{i_8} \d x_{i_9}
%  \d x_{i_{10}}
% $$
% $$
% = \frac{1}{4} \int_0^1 \int_0^{x_{i_{10}}}
%  \int_0^{x_{i_9}} \int_0^{x_{i_8}}
%  \int_0^{x_8}
%   \frac{x_5^7}{84} - \frac{x_5^8}{112 x_8}
%   - \frac{x_5^9}{432 x_5^2} + \frac{x_5^9}{216 x_8 x_5}
%   \d x_5 \d x_8 \d x_{i_8} \d x_{i_9}
%  \d x_{i_{10}}
% $$
% $$
% = \frac{1}{12096} \int_0^1 \int_0^{x_{i_{10}}}
%  \int_0^{x_{i_9}} \int_0^{x_{i_8}}
%   \frac{29 x_8^8}{8} - \frac{13 x_8^9}{9 x_8}
%   \d x_8 \d x_{i_8} \d x_{i_9}
%  \d x_{i_{10}}
% $$
% $$
% = \frac{1}{12096 \cdot 72} \int_0^1 \int_0^{x_{i_{10}}}
%  \int_0^{x_{i_9}} \int_0^{x_{i_8}}
%   {157 x_8^8}
%   \d x_8 \d x_{i_8} \d x_{i_9}
%  \d x_{i_{10}} = \frac{157 \cdot 8!}{870912 \cdot 12!}
%    = \frac{785}{108 \cdot 12!}.
% $$

There are $3!$ orderings of the remaining variables, and
$\X_3$ and $\X_7$ can be again be switched, giving 12
orderings in this subcase.

\begin{flushleft}
\underline{Subcase 4iv}: The seven smallest variables are
$\X_1, \X_2, \X_3, \X_4, \X_7$, $\X_5$ and $\X_8$.
Unlike the previous cases, we have $\X_j=\X_4$, so we
restart with just the inner 3 integrals evaluated
at the top of the section:
\end{flushleft}
$$
\int_0^1 \int_0^{x_{i_{10}}}
 \int_0^{x_{i_9}} \int_0^{x_{i_8}} \int_0^{x_8}
 \int_0^{x_5} \int_0^{x_7}
 \frac{x_4^4}{4}
  \left[ \frac{x_7}{2} - \frac{x_7^2}{2 x_8}
  - \frac{x_7^3}{6 x_5^2} + \frac{x_7^3}{3 x_8 x_5}
  \right]
  \d x_4 \d x_7 \d x_5 \d x_8 \d x_{i_8} \d x_{i_9}
 \d x_{i_{10}}
$$
$$
= \frac{1}{20}\int_0^1 \int_0^{x_{i_{10}}}
 \int_0^{x_{i_9}} \int_0^{x_{i_8}} \int_0^{x_8}
 \int_0^{x_5}
  \left( \frac{x_7^6}{2} - \frac{x_7^7}{2 x_8}
  - \frac{x_7^8}{6 x_5^2} + \frac{x_7^8}{3 x_8 x_5} \right)
  \d x_7 \d x_5 \d x_8 \d x_{i_8} \d x_{i_9}
 \d x_{i_{10}}
$$
% $$
% = \frac{1}{20}\int_0^1 \int_0^{x_{i_{10}}}
%  \int_0^{x_{i_9}} \int_0^{x_{i_8}} \int_0^{x_8}
%   \frac{x_5^7}{14} - \frac{x_5^8}{16 x_8}
%   - \frac{x_5^9}{54 x_5^2} + \frac{x_5^9}{27 x_8 x_5}
%   \d x_5 \d x_8 \d x_{i_8} \d x_{i_9}
%  \d x_{i_{10}}
% $$
$$
= \frac{1}{540}\int_0^1 \int_0^{x_{i_{10}}}
 \int_0^{x_{i_9}} \int_0^{x_{i_8}} \int_0^{x_8}
  \left(
  \frac{10 x_5^7}{7} - \frac{11 x_5^8}{16 x_8} \right)
  \d x_5 \d x_8 \d x_{i_8} \d x_{i_9}
 \d x_{i_{10}}
$$
% $$
% = \frac{1}{540}\int_0^1 \int_0^{x_{i_{10}}}
%  \int_0^{x_{i_9}} \int_0^{x_{i_8}}
%   \frac{10 x_8^8}{56} - \frac{11 x_8^9}{144 x_8}
%   \d x_8 \d x_{i_8} \d x_{i_9}
%  \d x_{i_{10}}
% $$
$$
= \frac{1}{540}\int_0^1 \int_0^{x_{i_{10}}}
 \int_0^{x_{i_9}} \int_0^{x_{i_8}}
  \frac{103 x_8^8}{1008}
  \d x_8 \d x_{i_8} \d x_{i_9}
 \d x_{i_{10}}
= \frac{103 \cdot 8!}{544320 \cdot 12!} = \frac{206}{27 \cdot 12!}.
$$
There are $3!$ orderings of the remaining variables,
and $\X_3$ cannot be interchanged with $\X_7$ due to
the interceding $\X_4$.

\begin{flushleft}
\underline{Subcase 4v}: The eight smallest variables are
$\X_1, \X_2, \X_3, \X_7, \X_5$, $\X_k$, $\X_l$ and $\X_8$.
\end{flushleft}
This follows Subcase 4ii until we arrive at integrand
$\frac{1}{32} \left[\frac{x_k^7}{14} - \frac{5 x_k^8}{168 x_8}\right]$.
We continue through integrands
$$\frac{1}{32} \left[ \frac{x_l^8}{112} - \frac{5 x_l^9}{1512 x_8}\right]
\text{ and } \frac{x_8^9}{96768}$$
and eventually to $\frac{4}{15 \cdot 12!}$.
% $$
% \frac{1}{32} \int_0^1 \int_0^{x_{i_{10}}}
%  \int_0^{x_{i_9}} \int_0^{x_8}
%   \frac{x_l^8}{112} - \frac{5 x_l^9}{1512 x_8}
%   \d x_l \d x_8 \d x_{i_9}
%  \d x_{i_{10}}
% $$
% $$
% = \frac{1}{32} \int_0^1 \int_0^{x_{i_{10}}}
%  \int_0^{x_{i_9}}
%   \frac{x_8^9}{1008} - \frac{5 x_8^{10}}{15120 x_8}
%   \d x_8 \d x_{i_9}
%  \d x_{i_{10}}
% $$
% $$
% = \frac{1}{96768} \int_0^1 \int_0^{x_{i_{10}}}
%  \int_0^{x_{i_9}}
%   {x_8^9}
%   \d x_8 \d x_{i_9}
%  \d x_{i_{10}} = \frac{9!}{96768 \cdot 12!} = \frac{4}{15 \cdot 12!}.
% $$

There are two orderings of the remaining variables, $(k,l)$ can
be $(4,6)$ or $(6,4)$ and $\X_3$ and $\X_7$ may be reversed,
giving a total of 8 orderings in this subcase.

\begin{flushleft}
\underline{Subcase 4vi}: The eight smallest variables are
$\X_1, \X_2, \X_3, \X_7, \X_4$, $\X_5$, $\X_6$ and $\X_8$.
\end{flushleft}

This follows Subcase 4iii until we arrive at the integrand
$\frac{1}{12096} \left[ {29 x_5^7} - \frac{13 x_5^8}{x_8} \right]$.
Two more integrations bring integrands:
$$
\frac{1}{12096} \left[ \frac{29 x_6^8}{8} - \frac{13 x_6^9}{9 x_8} \right]
\text{ and } \frac{1}{12096}\left[ \frac{31 x_8^9}{120}\right].
$$
Continuing we get $\frac{31}{4 \cdot 9!}$.
% $$
% \frac{1}{12096} \int_0^1 \int_0^{x_{i_{10}}}
%  \int_0^{x_{i_9}} \int_0^{x_8}
%   \frac{29 x_6^8}{8} - \frac{13 x_6^9}{9 x_8}
%   \d x_6 \d x_8 \d x_{i_9}
%  \d x_{i_{10}}
% $$
% $$
% = \frac{1}{12096} \int_0^1 \int_0^{x_{i_{10}}}
%  \int_0^{x_{i_9}}
%   \frac{29 x_8^9}{72} - \frac{13 x_8^{10}}{90 x_8}
%   \d x_8 \d x_{i_9}
%  \d x_{i_{10}}
% $$
% $$
% = \frac{1}{12096} \int_0^1 \int_0^{x_{i_{10}}}
%  \int_0^{x_{i_9}}
%   \frac{31 x_8^9}{120}
%   \d x_8 \d x_{i_9}
%  \d x_{i_{10}} = \frac{31 \cdot 9!}{1451520 \cdot 12!}
%   = \frac{31}{4 \cdot 9!}.
% $$
There are two orderings of the remaining variables, and
$\X_3$ and $\X_7$ can be reversed, giving a total of 4
orderings in this subcase.

\begin{flushleft}
\underline{Subcase 4vii}: The eight smallest variables are
$\X_1, \X_2, \X_3, \X_4, \X_7$, $\X_5$, $\X_6$ and $\X_8$.
\end{flushleft}
This follows Subcase 4iv until we arrive at the integrand
$\frac{1}{540} \left[\frac{10 x_5^7}{7} - \frac{11 x_5^8}{16 x_8} \right]$.
Continuing, we see:
$$
\frac{1}{540} \left[  \frac{10 x_6^8}{56} - \frac{11 x_6^9}{144 x_8} \right]
\text{ and } \frac{1}{540} \left[\frac{123 x_8^9}{10080} \right]
$$
on our way to $\frac{41}{5 \cdot 12!}$.
% $$
% \frac{1}{540}\int_0^1 \int_0^{x_{i_{10}}}
%  \int_0^{x_{i_9}} \int_0^{x_8}
%   \frac{10 x_6^8}{56} - \frac{11 x_6^9}{144 x_8}
%   \d x_6 \d x_8 \d x_{i_9}
%  \d x_{i_{10}}
% $$
% $$
% = \frac{1}{540}\int_0^1 \int_0^{x_{i_{10}}}
%  \int_0^{x_{i_9}}
%   \frac{10 x_8^9}{504} - \frac{11 x_8^{10}}{1440 x_8}
%   \d x_8 \d x_{i_9}
%  \d x_{i_{10}}
% $$
% $$
% = \frac{1}{540}\int_0^1 \int_0^{x_{i_{10}}}
%  \int_0^{x_{i_9}}
%   \frac{123 x_8^9}{10080}
%   \d x_8 \d x_{i_9}
%  \d x_{i_{10}}= \frac{41 \cdot 9!}{540 \cdot 3360 \cdot 12!}
%    = \frac{41}{5 \cdot 12!}.
% $$
In this subcase, $\X_3$ and $\X_7$ cannot be interchanged, and
there are 2 orderings of the remaining 2 variables.

This completes Case 4, which contains the remaining 104
possible orderings of the variables,  and the terms
in (\ref{eq:pieces}) corresponding to these
orderings of the variables have total weight per symmetry
class of $\frac{3627}{5 \cdot 12!}$.

Summing over the four cases, the contributions of
$\frac{13140}{12!}$, $\frac{235797}{35 \cdot 12!}$,
$\frac{6651}{7 \cdot 12!}$ and $\frac{3627}{5 \cdot 12!}$
respectively, give a total of
$\frac{107763}{5 \cdot 12!}$ summed over the 5040 symmetry
class representatives in (\ref{eq:pieces}).
Multiplying by the $720$ symmetries of the variables, we
find that
$\eu(5)=\frac{35921}{1108800}=0.032396\overline{284271}$.

\bigskip

%\begin{remark}
As noted in Section~\ref{se:recur},
the computed value of $\eu(5)$ is used in a lower
bound for $\eu(n)$ for $n \ge 5$.
%\end{remark}

%\begin{remark}
It is easy to compute the first few digits of this number
by simulating the ten uniform random variables.
A short MATLAB program sampled the ten variables $10^{12}$ times
and computed the minimum, the computed number agreed
with our calculation to the 7th decimal place.
While this number arises in a relatively simple way,
we do not know of it arising in other places.
%\end{remark}

Unfortunately, it would be much harder to use such a simulation
to get approximate values of $\eu(6)$ or $\eu(7)$. The proof method
used above for $K_5$ does not generalize to $K_6$ or $K_7$ either.
To simulate $\eu(6)$ we would
need to catalogue the minimal ways of drawing $\K_6$,
i.e.~drawings $\D$ for which $\XX(\D)$ is inclusion-wise minimal.

%
%%%%% 4. OTHER DISTRIBUTIONS %%%%%%%%%%%%%%%%%%%%%%%%%%%%%%%%%%%%%
%
\section{Other distributions and moments} %%%%%%%%%%%

For some other very simple discrete probability distributions,
it is possible to compute the expected crossing number
exactly.  Here we do this for the simplest example and conclude that the
first two moments of the distribution do not determine
the expectation.

Consider for $0 \le t < u$, the discrete distribution where edges
have weight $t$ or $u$ with probability $\frac{1}{2}$.
Let $\edisc(n,t,u)$ be the expected weighted crossing number
of $\K_n$ with the distribution for given $t,u$; if the
parameter $u$ is omitted we will assume it is $1-t$.
Then it is easy to see that
$$
\crn(\K_5,w)= \left\{
 \begin{array}{ll}
   t^2 & \text{if there is a pair of non-adjacent edges of weight } t \\
   u^2 & \text{if all edges have weight } u \\
   tu & otherwise.
 \end{array}
\right.
$$

All $2^{10}$ possible assignments of $t$'s and $u$'s to the edges
are equally likely.  There is only one way for all edges to have
weight $u$.  Otherwise, if we do not have two non-adjacent edges of weight
$t$, we must either have all edges of weight $t$ incident with a single vertex, or
three edges forming a triangle.  In the former situation, we may
have one edge (10 assignments), two edges (30 assignments),
three edges (20 assignments) or four edges (5 assignments).
For the triangles, we have 10 more assignments.  The remaining
948 assignments of $t$'s and $u$'s to the edges have a pair
of non-adjacent edges of weight $t$.
Therefore, $\edisc(5,t,u)=\frac{1}{1024}(948 t^2 + 75 tu + u^2)$,
which simplifies to
$\edisc(5,t)=\frac{1}{1024}(874 t^2 + 73t +1)$
when $u=1-t$.

The mean and variance of the considered discrete distribution are
$\frac{u+t}{2}$ and $\frac{(u-t)^2}{4}$, respectively.
If we take $u=1-t$,
then the mean is $\frac{1}{2}$, matching the the mean of
the uniform distribution, while the variance is $\frac{(1-2t)^2}{4}$.
Since the variance of the uniform case is $\frac{1}{12}$, by choosing
$t=\frac{3-\sqrt{3}}{6}$, we get a distribution that matches
the uniform distribution in its first two moments.
However the above calculation shows that
$$\edisc(5,\frac{3-\sqrt{3}}{6})=\frac{1973-947\sqrt{3}}{6144}
 \approx 0.05416 > \eu(5).$$

We conclude that the first two moments of the input distribution
on the edges are not sufficient to determine the expected crossing
number. We believe that a constant number of higher moments is
not sufficient either. Perhaps, up to ${n \choose 2}$ moments are required.

%
%%%%% 5. ASYMPTOTICS %%%%%%%%%%%%%%%%%%%%%%%%%%%%%%%%%%%%%
%
\section{Asymptotics}\label{se:asym}

Some standard arguments used for crossing number estimates work also for
the expectations.
In this section we show that simple adaptations of these arguments
show that $\eu(n)$ is $\Theta(n^4)$.  Since $\crn(\K_n)$ is
$O(n^4)$ and an upper bound for $\eu(n)$, we need only show the
lower bound.  We remark that the asymptotic upper bound $\crn(\K_n)$
can be obtained trivially from the fact that there are only
$O(n^4)$ pairs of edges in $\K_n$, but that much better constructive
upper bounds exist and are an ongoing research challenge, see for
instance \cite{AAK,PR07}.

\subsection{Asymptotics via a recurrence}\label{se:recur}

We recall that we denote the crossing weight of a given
drawing $\D$ of a graph weighted by $w$ as $\crn(\D,w)$,
and the weighted crossing number of $G$ weighted by $w$
(i.e.~the minimum over all drawings) by $\crn(G,w)$.

Given a drawing $\D$ of $\K_n$ with weights $w$, we can consider 
the induced drawings of copies of $\K_n-v \approx \K_{n-1}$ obtained
by removing each vertex $v\in V = V(K_n)$ from $\K_n$ in turn.  Then
\begin{equation}
\sum_{v \in V} \crn(\D|_{\K_n-v},w|_{\K_n-v}) =
   (n-4) \crn(\D, w)
   \label{eq:recur}
\end{equation}
since each pair of disjoint edges $ij, i'j'$ of $\K_n$ appear
in all but four of the terms on the left side of (\ref{eq:recur}).

Now consider $\K_n$ for $n>4$ with a fixed weighting $w$.
There is some optimal drawing $\D^*$ of $\K_n$ such that
$\crn(\K_n,w)=\crn(\D^*,w)$.
% For each $v \in V$, let $\D_v$ be a drawing of $\K_{n-1}$
% which attains $\crn(\K_{n-1},w|_{\K_n \setminus \{v\}})$.
Now:
\begin{align*}
\crn(\K_n,w) & =\crn(\D^*,w) = \frac{1}{n-4} \sum_{v \in V}
  \crn(\D^*|_{\K_n-v},w|_{\K_n-v}) \\
  & \ge \frac{1}{n-4} \sum_{v \in V}
  \min_{\D} \crn(\D|_{\K_n-v},w|_{\K_n-v})
 = \frac{1}{n-4} \sum_{v \in V}
  \crn(\K_n-v,w|_{\K_n-v}).
\end{align*}
If the weights in $w$ are i.i.d.~random variables, we can take
expectations on both sides to get $\eu(n) \ge \frac{n}{n-4} \eu(n-1)$.
Applying this inequality recursively, we find for $n \ge 6$ that
$\eu(n) \ge \tfrac{1}{5}{n \choose 4} \eu(5)$.

\subsection{Asymptotics via the Crossing Lemma}

The following result, known as the Crossing Lemma, was proved
independently by Ajtai et al. \cite{ACNS} and Leighton \cite{Lei}.
The version given below (with the specific constant $1024/31827 > 0.032$)
is due to Pach et al.~\cite{PRTT}.

\begin{theorem}[Crossing Lemma]
Let $G$ be a graph of order $n$ with $m \ge \frac{103}{16}n$ edges. Then
$$ \crn(G) \ge \frac{1024}{31827}\, \frac{m^3}{n^2}.$$
\end{theorem}

Let $\pi$ be a probability distribution with expectation $\E(\pi)=\mu$.
We define the {\em complementary probability distribution} $\pi^*$ by
setting $\pi^*(\mu+x) = \pi(\mu-x)$. For the purpose of the following argument,
let us assume that our probability distribution is symmetric, i.e.,
$\pi = \pi^*$. Then, given a random weight function $w$, the
{\em complementary weight function} $w^*$, defined as $w^*(e) = 2\mu-w(e)$,
has the same distribution as $w$. Let us define $w'$ to be either
$w$ or $w^*$, so that $w'(e)\ge \mu$ holds for at least half of the
edges $e\in E(G)$. Finally, let $w_1$ be defined as $w_1(e)=0$ if
$w'(e)<\mu$, and $w_1(e)=1$ if $w'(e)\ge\mu$. Since
$\crn(G,w)+\crn(G,w^*) \ge \crn(G,w') \ge \mu^2\crn(G,w_1)$,
the following holds:
\begin{eqnarray*}
   \E(\crn(G,w)) &=& \tfrac{1}{2}\E(\crn(G,w)+\crn(G,w^*))
     \ge \tfrac{1}{2}\E(\crn(G,w')) \\
     &\ge& \frac{\mu^2}{2}\E(\crn(G,w_1))
     \ge \frac{\mu^2}{2}\cdot \frac{1024}{31827}\, \frac{(m/2)^3}{n^2}
     = \frac{64\mu^2}{31827}\, \frac{m^3}{n^2}.
\end{eqnarray*}
This gives a version of the crossing lemma for expectations.
With a little more care we can improve the above bound
and also get rid of the symmetry condition. In order to do this, we
replace the mean by the {\em median}, i.e.\ the largest number
$\nu$ such that $Prob[w(e)\ge \nu]\ge \frac{1}{2}$.

\begin{theorem}[Crossing Lemma for Expectations]
Let $G$ be a graph of order $n$ with $m \ge \frac{103}{16(1-4^{-1/3})}\,n$ edges.
Suppose that each edge $e\in E(G)$ gets a random weight $w(e)$,
where the weights of distinct edges are independent non-negative
random variables (not necessarily i.d.) whose median is at least $\nu>0$. Then
$$ \E(\crn(G,w)) \ge \frac{128\nu^2}{31827}\cdot \frac{m^3}{n^2}.$$
\end{theorem}

\begin{proof}
Given $w$, we introduce related weights $w''$ and $w_2$ in a similar
(but not identical) way as above: we let $w''(e)=0$ if
$w(e)<\nu$, and $w''(e)=\nu$ if $w(e)\ge\nu$; we let $w_2(e)=w''(e)/\nu$
be the corresponding weight with values 0 and 1. Note that
$Prob[w''(e)=\nu]\ge \frac{1}{2}$ and $Prob[w_2(e)=1]\ge \frac{1}{2}$.
Similarly as before, we have $\crn(G,w) \ge \crn(G,w'') = \nu^2\crn(G,w_2)$.

Note that $w$ determines a spanning subgraph $F_w\subseteq G$, whose edges are those
 edges of $G$ for which $w_2(e)=1$. The graph $F_w$ is a random subgraph of $G$,
and for each spanning subgraph $F$ of $G$ we let $Prob(F)$ be the probability that
$F=F_w$. We will need a lower bound for the
sum $\sum \crn(F)Prob(F)$ taken over all (spanning) $F\subseteq G$.
To do this, let us define $F'\subseteq F$ as the spanning subgraph of $G$
such that $e\in E(F')$ if $w(e)\ge \nu_e\ge\nu$, where $\nu_e$ is the median of 
the random variable $w(e)$. 
The threshold case when $w(e)=\nu_e$ is to be considered so that
$Prob[e\in E(F')]=\frac{1}{2}$. Then $F'$ is also a random spanning subgraph of $G$ and
$Prob(F')=2^{-m}$. Since the event that an edge $e$ is in $F'$ is contained in the
event that $e\in E(F)$, we have for each $F$
$$Prob(F) = \sum_{F'\subseteq F} \alpha(F,F') Prob(F'),$$
where $\alpha(F,F')\ge0$ is the probability that we have $F_w=F$ under the condition
that $F'$ is given.
Clearly, $\sum_{F\supseteq F'} \alpha(F,F')=1$ for every fixed $F'$.
Since $\crn(F')$ is an increasing function, we have:
\begin{eqnarray*}
  \sum_{F\subseteq G} \crn(F)Prob(F) 
     &=& \sum_{F\subseteq G} \crn(F)\sum_{F'\subseteq F} \alpha(F,F') Prob(F')\\
     &\ge& \sum_{F'\subseteq G} \crn(F') Prob(F') \sum_{F\supseteq F'} \alpha(F,F') \\
     &=& \sum_{F'\subseteq G} \crn(F')Prob(F')
     = \sum_{F'\subseteq G} 2^{-m}\crn(F').
\end{eqnarray*}

We will employ another notion:
$$
   \lambda(k,n) = \min\{ \crn(F)\mid |V(F)|=n, |E(F)|=k\}.
$$
By the Crossing Lemma, $\lambda(k,n) \ge \frac{1024}{31827}\frac{k^3}{n^2}$
if $k\ge \tfrac{103}{16}n$.
Using the introduced quantities, we obtain the following estimate:
\begin{eqnarray*}
  \E[\crn(G,w)] &=& \int_{\R^E} \crn(G,w) \d w
     \ge \int_{\R^E} \crn(G,w') \d w = \nu^2 \int_{\R^E} \crn(G,w_1) \d w \\
     &\ge& \nu^2 \sum_{F\subseteq G} \crn(F)Prob(F)
     \ge \nu^2 \sum_{F'\subseteq G} 2^{-m}\crn(F')\\
     &\ge& \nu^2 \sum_{k=0}^m ~ \sum_{F\subseteq G, |E(F)|=k} 2^{-m}\crn(F) \\
     &\ge& \nu^2 \, 2^{-m} \sum_{k=0}^m \binom{m}{k} \lambda(k,n) \\
     &\ge& \frac{1024\nu^2}{31827\cdot 2^m n^2}
           \sum_{k=\lceil 103n/16\rceil}^m \binom{m}{k} k^3.
\end{eqnarray*}
Note that we have $k^3 + (m-k)^3 \ge \tfrac{1}{4}m^3$ for $0\le k\le m$,
and that for $k < 103n/16$, we have
$(m-k)^3 \ge (m-103n/16)^3 \ge (m-(1-4^{-1/3})m)^3 = \tfrac{1}{4}m^3$.
Thus,
$$
   \sum_{k=\lceil 103n/16\rceil}^m \binom{m}{k} k^3 \ge
   \tfrac{1}{2}\sum_{k=0}^m \binom{m}{k} \tfrac{1}{4}m^3 = \tfrac{1}{8}2^m m^3.
$$
The above inequalities imply:
$$
   \E[\crn(G,w)]
     \ge \frac{1024\nu^2}{31827\cdot 2^m n^2} \,\frac{1}{8}2^m m^3
      =  \frac{128 \nu^2 \cdot m^3}{31827\cdot n^2}
$$
which we were to prove.
\end{proof}

%%%%% 7. ACKNOWLEDGMENTS %%%%%%%%%%%%%%%%%%%%%%%%%%%%%%%%%%%%%%%%%%%%%%%%%%%%%%
\subsection*{Acknowledgments}
B.M. was supported in part by an NSERC Discovery Grant (Canada), by the Canada Research Chair program, and by the Research Grant P1--0297 of ARRS (Slovenia). He is on leave from: IMFM \& FMF, Department of Mathematics, University of Ljubljana, Ljubljana, Slovenia.
T.S. was supported in part by an NSERC Discovery Grant.
The authors are grateful to Luis Goddyn for some helpful discussions on the subject.

%
% BIBLIOGRAPHY
%
% Setup for bibtex
% \nocite{*}  % bib entries to include those that are not actually cited.
\bibliographystyle{amsalpha}
% End setup for bibtex.

\bibliography{refs}

\providecommand{\bysame}{\leavevmode\hbox to3em{\hrulefill}\thinspace}
\providecommand{\MR}{\relax\ifhmode\unskip\space\fi MR }
\providecommand{\MRhref}[2]{%
  \href{http://www.ams.org/mathscinet-getitem?mr=#1}{#2}
}
\providecommand{\href}[2]{#2}
\begin{thebibliography}{PRTT06}

\bibitem[AAK06]{AAK}
O.~Aichholzer, F.~Aurenhammer, and H.~Krasser, \emph{On the crossing number of
  complete graphs}, Computing \textbf{76} (2006), 165--176.

\bibitem[ACNS82]{ACNS}
M.~Ajtai, V.~Chv{\'a}tal, M.~M. Newborn, and E.~Szemer{\'e}di,
  \emph{Crossing-free subgraphs}, Theory and practice of combinatorics,
  North-Holland Math. Stud., vol.~60, North-Holland, Amsterdam, 1982,
  pp.~9--12.

\bibitem[DN03]{DN03}
H.~A. David and H.~N. Nagaraja, \emph{Order statistics}, third ed., Wiley
  Series in Probability and Statistics, Wiley-Interscience, 2003.

\bibitem[Guy72]{Guy}
Richard~K. Guy, \emph{Crossing numbers of graphs}, Graph theory and
  applications ({P}roc. {C}onf., {W}estern {M}ichigan {U}niv., {K}alamazoo,
  {M}ich., 1972, Lecture Notes in Math., Vol. 303, Springer, Berlin, 1972,
  pp.~111--124.

\bibitem[Lei84]{Lei}
Frank~Thomson Leighton, \emph{New lower bound techniques for {VLSI}}, Math.
  Systems Theory \textbf{17} (1984), 47--70.

\bibitem[Moh08]{Mo1}
Bojan Mohar, \emph{Crossing numbers of graphs on the plane and on other
  surfaces}, Abstracts of the 20th Workshop on Topological Graph Theory in
  Yokohama, Nov.~25 to 28, 2008, Yokohama, Japan, 2008.

\bibitem[Moh10]{Mo2}
\bysame, \emph{Do we really understand the crossing numbers?}, Mathematical
  foundations of computer science 2010. 35th international symposium, MFCS
  2010, Brno, Czech Republic, August 23--27, 2010. Proceedings. Hlin\v{e}n\'y,
  Petr et al. (ed.), Lecture Notes in Computer Science 6281, Springer, Berlin,
  2010, pp.~38--41.

\bibitem[PR07]{PR07}
Shengjun Pan and R.~Bruce Richter, \emph{The crossing number of {$K_{11}$} is
  $100$}, J. Graph Theory \textbf{56} (2007), 128--134.

\bibitem[PRTT06]{PRTT}
J{\'a}nos Pach, Rado{\v{s}} Radoi{\v{c}}i{\'c}, G{\'a}bor Tardos, and G{\'e}za
  T{\'o}th, \emph{Improving the crossing lemma by finding more crossings in
  sparse graphs}, Discrete Comput. Geom. \textbf{36} (2006), 527--552.

\bibitem[RS09]{RS09}
R.~Bruce Richter and G.~Salazar, \emph{Crossing numbers}, Topics in topological
  graph theory, Encyclopedia Math. Appl., vol. 128, Cambridge Univ. Press,
  Cambridge, 2009, pp.~133--150.

\bibitem[SSV95]{SSV95}
Farhad Shahrokhi, L\'aszl\'o Sz\'ekely, and Imrich Vrt'o, \emph{Crossing
  numbers of graphs, lower bound techniques and algorithms: A survey}, Graph
  Drawing (Roberto Tamassia and Ioannis Tollis, eds.), Lecture Notes in
  Computer Science, vol. 894, Springer, 1995, pp.~131--142.

\bibitem[Vrt10]{Vrto}
Imrich Vrt'o, \emph{Crossing numbers of graphs: A bibliography}, 2010,
  available at: {\tt http://www.ifi.savba.sk/{\textasciitilde}imrich}.

\end{thebibliography}

\end{document}